\documentclass[11pt]{article}
\usepackage{authblk}
\usepackage{amssymb}
\usepackage[all]{xy}
\usepackage{mathrsfs}
\usepackage{amsmath,amssymb,amsthm}
\usepackage{extarrows}
\newtheorem{theorem}{Theorem}[section]

\newtheorem{remark}[theorem]{Remark}

\DeclareMathAlphabet{\mathpzc}{OT1}{pzc}{m}{it}
\def\to{\rightarrow}
\def\ot{\otimes}
\def\f{\mathfrak}

\def\c{\mathcal}

\def\r{\mathrm}
\def\bb{\mathbb}
\def\s{\mathscr}

\def\ov{\overline}
\def\sl{\langle}
\def\sr{\rangle}
\date{}
\begin{document}
\title{On the Differential Geometry of Some Classes of\\ Infinite Dimensional Manifolds}
\author{Maysam Maysami Sadr\thanks{sadr@iasbs.ac.ir, corresponding author, orcid.org/0000-0003-0747-4180}
\\\vspace{2mm}\&\vspace{.5mm}\\Danial Bouzarjomehri Amnieh\thanks{danial.bouzarj@iasbs.ac.ir}}
\affil{Department of Mathematics,\\ Institute for Advanced Studies in Basic Sciences,\\Zanjan, Iran }
\maketitle
\begin{abstract}
Albeverio, Kondratiev, and R\"{o}ckner have introduced a type of differential geometry, which we call lifted geometry, for the configuration space
$\Gamma_X$ of any manifold $X$. The name comes from the fact that various elements of the geometry of $\Gamma_X$ are constructed via lifting of
the corresponding elements of the geometry of $X$. In this note, we construct a general algebraic framework for lifted geometry which
can be applied to various ``infinite dimensional spaces'' associated to $X$.
In order to define a lifted geometry for a ``space'', one dose not need any topology or local coordinate system on the space.
As example and application, lifted geometry for spaces of Radon measures on $X$, mappings into $X$, embedded submanifolds of $X$,
and tilings on $X$, are considered. The gradient operator in the lifted geometry of Radon measures is considered.
Also, the construction of a natural Dirichlet form associated to a Random measure is discussed.
It is shown that Stokes' Theorem appears as ``differentiability'' of ``boundary operator'' in the lifted geometry of spaces of submanifolds.
It is shown that (generalized) action functionals associated with Lagrangian densities on $X$
form the algebra of smooth functions in a specific lifted geometry for the path-space of $X$.

\textbf{MSC 2020.} 58B99, 58D10, 58D15, 58B10, 46T05.

\textbf{Keywords.} Algebraic differential geometry, infinite dimensional manifold, smooth function, vector field, differential form.
\end{abstract}
\section{Introduction}\label{2110300801}
Albeverio, Kondratiev, and R\"{o}ckner in \cite{Albeverio0,Albeverio1,Albeverio2} defined a type of differential geometry for the configuration
space $\Gamma_X$ of a smooth manifold $X$. ($\Gamma_X$ is the set of all locally finite subsets of $X$
and may be identified with the set of all Radon measures on $X$ of the form $\Sigma_{x\in S}\delta_x$ where $S$ is a countable subset
of $X$ without any limit point.) Their main idea was to construct vector fields, differential forms, metrics, and other basic objects
of the geometry of $\Gamma_X$, via lifting, in a certain meaning, of the corresponding objects on $X$. (Note that $\Gamma_X$ is not modeled on a
single topological linear space and hence the ordinary differential geometry of (infinite dimensional) manifolds (\cite{Lang1,Michor1})
can not be applied to it.) See \cite{Albeverio3,Albeverio4,Finkelshtein2,Finkelshtein1,Kondratiev1,Kuchling1,Ma1,Privault1,Rockner1}
for (some what) the same idea and its applications.

The main goal of this note is to introduce, in an axiomatic and algebraic way, a type of differential geometry called
\emph{lifted geometry} that generalizes the mentioned geometry of $\Gamma_X$ to a rather large class of spaces and infinite dimensional manifolds
associated with $X$.

In $\S$\ref{2110300802} we describe basic elements of a type of differential geometry in an abstract algebraic framework.
The similar geometries have been considered by many authors, see for instance \cite{Dubois-Violette1} and \cite{Garcia-Bondia1}.
In $\S$\ref{2110300802-1} we recall some algebraic preliminaries. In $\S$\ref{2111120808} we define a \emph{geometry} to be a pair $(A,D)$
where $A$ is a commutative real-algebra and $D$ is a Lie-algebra of derivations on $A$. The new aspect of our geometry (rather than the similar
concepts introduced by others) is that it is developed with respect to an arbitrary Lie-algebra $D$ of derivations instead the Lie-algebra of all derivations. (As we will see, this key property enables us to apply effectively the geometry for our favorite (infinite dimensional) manifolds.)
In $\S$\ref{2110300802-3} and $\S$\ref{2110300802-4} we define respectively differential forms and de Rham cohomology in our geometry.
In $\S$\ref{2111180700} a \emph{geometry on a set $\s{S}$} is defined to be a geometry $(A,D)$ such that $A$ is an algebra of real functions
on $\s{S}$. Then the concepts of tangents space, vector fields, geometrization of differential forms, and (weakly) differentiable mappings
are considered.

In $\S$\ref{2110300803} we introduce a general framework for lifted geometry. Roughly speaking, let $\s{S}$ be a set of objects \emph{associated
appropriately} with a smooth manifold $X$ such that any (complete) vector field on $X$ induces in a \emph{natural way} a flow on $\s{S}$. Then
a lifted geometry for $\s{S}$ is a geometry $(A,D)$ on the set $\s{S}$ such that $A$ is an algebra of functions on $\s{S}$ obtained via a \emph{distinguished} lifting procedure of smooth functions (or differential forms) on $X$, and such that derivations in $D$ are also constructed
by lifting of vector fields of $X$ in a canonical way. The algebra $A$ is interpreted as the algebra of smooth
functions on $\s{S}$ and $D$ as the Lie-algebra of smooth vector fields on $\s{S}$. Significance of any lifted geometry for $\s{S}$ is that it is
constructed without any using of local chart or even topology on $\s{S}$.

In the rest sections we consider various examples of lifted geometry and its applications:
In $\S$\ref{2110300804} we extend some contents of \cite{Albeverio1} and construct a lifted geometry for the space $\s{M}_X$
of Radon measures on $X$, and its suitable subsets. In $\S$\ref{2110300804.5} we consider construction of gradient operator in the
lifted geometry of $\s{M}_X$ where $X$ is a Riemannian manifold. Also there is a little discussion about the corresponding Dirichlet form
associated with a random measure.
In $\S$\ref{2110300805} we consider a lifted geometry for the set $\s{F}_X^Y$ (and its suitable subsets)
of measurable mappings from a measurable space $Y$ into $X$. In $\S$\ref{2110300806} we construct
a lifted geometry for the set $\s{E}_X^k$ of $k$-dimensional embedded submanifolds of $X$. Also we show that in our framework stokes' theorem
may be interpreted as \emph{differentiability} of the boundary operator $\partial:\s{E}_X^k\to\s{E}_X^{k-1}$.
In $\S$\ref{2110300807} we construct a lifted geometry for the set $\s{T}_X$ of \emph{tilings} of $X$. In $\S$\ref{2110300808}
we consider a lifted geometry for the set $\s{C}_X$ of smooth curves in $X$ such that its algebra
is defined to be the algebra of generalized action functionals associated to Lagrangian densities on $X$.

Although, in $\S\S$\ref{2110300804}-\ref{2110300808}, in each case we describe only one type of lifted geometry
but the reader will recognize that our methods can be appropriately modified to produce various lifted geometries for the mentioned spaces.

\textbf{Notations.} For a smooth manifold $X$ we denote by $\r{C}^\infty(X)$ the algebra of smooth real-valued functions on $X$. $\Omega^n(X)$
and $\r{Vec}(X)$ respectively denote the $\r{C}^\infty(X)$-module of $n$-differential forms and  the Lie-algebra of vector fields on $X$. The
subset of functions with compact support is denoted by $\r{C}^\infty_\r{c}(X)$. Similarly, $\Omega^n_\r{c}(X)$ and $\r{Vec}_\r{c}(X)$
denote the subsets of forms and vector fields with compact support. The Lie-derivative $\r{w.r.t.}$ $v\in\r{Vec}(X)$ is denoted by $\r{d}_v$.
The exterior-derivative is denoted by $\r{d}$.
\section{A Differential Geometry for Commutative Algebras}\label{2110300802}
\subsection{Preliminaries}\label{2110300802-1}
Throughout all vector spaces and algebras are over the real field $\bb{R}$.
Algebras have unit and modules are unital. Algebra morphisms preserve the units.
For vector spaces $V,W$, the vector space of linear mappings from $V$ into $W$ is denoted by $\r{Lin}(V,W)$. Composition of linear
mappings makes $\r{Lin}(V):=\r{Lin}(V,V)$ into an algebra. We consider $\r{Lin}(V)$ also as a Lie-algebra with the canonical
bracket $[a,a']:=aa'-a'a$ ($a,a'\in\r{Lin}(V)$).
Let $A$ be a commutative algebra. For $A$-modules $M,N$ the set of $A$-module morphisms from $N$ into $M$ is denoted by $\r{Mod}(N,M)$.
The vector space $\r{Lin}(V,M)$ is a $A$-module with the module operation induced by that of $M$ in the obvious way.
The vector space $\r{Mod}(N,M)$ is considered as a sub-$A$-module of $\r{Lin}(N,M)$. We let $\r{Mod}^0(N,M):=M$ and
$\r{Mod}^1(N,M):=\r{Mod}(N,M)$. We also let $\r{Mod}^n(N,M):=\r{Mod}(N^{\ot_A^n},M)$ ($n\geq2$) and consider it as a sub-$A$-module
of $\r{Lin}(N^{\ot_A^n},M)$. (Here $\ot_A$ denotes the tensor product of $A$-modules.) We denote by $\Lambda^n(N,M)\subset\r{Mod}^n(N,M)$
the sub-module of alternating morphisms i.e. the morphisms $\omega\in\r{Mod}^n(N,M)$ satisfying
$$\omega(x_{\sigma(1)}\ot \cdots\ot x_{\sigma(n)})=\r{sgn}(\sigma)\omega(x_1\ot\cdots\ot x_n).$$
We let $\r{Alt}:\r{Mod}^n(N,M)\to\Lambda^n(N,M)$ be the $A$-module morphism defined by
$$\r{Alt}(f)(x_1\ot\cdots\ot x_n):=\frac{1}{n!}\sum_\sigma\r{sgn}(\sigma) f(x_{\sigma(1)}\ot \cdots\ot x_{\sigma(n)}).$$
Thus $\r{Alt}$ is a left inverse for the inclusion $\Lambda^n(N,M)\hookrightarrow\r{Mod}^n(N,M)$.
We have the exterior-algebra $\Lambda^*(N,A):=\oplus_{n=0}^\infty\Lambda^n(N,A)$ with the wedge product
$$\omega\wedge\eta:=\frac{(k+l)!}{k!l!}\r{Alt}(\omega\ot\eta)\hspace{10mm}(\omega\in\Lambda^k(N,A),\eta\in\Lambda^l(N,A)).$$
A derivation $d:A\to M$ is a linear map satisfying $d(ab)=d(a)b+ad(b)$ for $a,b\in A$. The set of all derivations from $A$ to $M$ is denoted
by $\r{Der}(A,M)$. This may be considered as a sub-$A$-module of $\r{Lin}(A,M)$. Also note that $\r{Der}(A):=\r{Der}(A,A)$ is a sub-Lie-algebra
of $\r{Lin}(A)$. For any graded-algebra $B=\oplus_{n=0}^\infty B_n$ a graded-derivation of degree $k\in\bb{Z}$ is a homogenous linear mapping
$d:B\to B$ of degree $k$ (i.e. $d(B_n)\subseteq B_{n+k}$) satisfying $d(ab)=d(a)b+(-1)^{kn}ad(b)$ for $a\in B_n$ and $b\in B$.
A differential graded-algebra is a pair $(B,d)$ where $B$ is a graded-algebra and $d$ is a graded-derivation on $B$ of degree
$1$ satisfying $d^2=0$. Then $\r{Ker}(d)$ is a subalgebra of $B$ and $\r{Img}(d)$ is an ideal in $\r{Ker}(d)$. The graded-algebra
$\r{Ker}(d)/\r{Img}(d)$ is called cohomology-algebra of $(B,d)$. Also the vector space $(\r{Ker}(d)\cap B_n)/(\r{Img}(d)\cap B_n)$ is
called $n$'th cohomology group of $(B,d)$. A graded-algebra $B$ is called graded-commutative if $ab=(-1)^{nm}ba$ for $a\in B_n$ and $b\in B_m$.
(Thus $\Lambda^*(N,A)$ is graded-commutative.)
\subsection{Algebraic differential geometries}\label{2111120808}
By a \emph{geometry} we mean a pair $\f{G}=(A,D)$ where $A$ is a commutative algebra and $D$ is
a sub-Lie-algebra of $\r{Der}(A)$ (not necessarily sub-$A$-module). For any geometry $\f{G}=(A,D)$ we let $\ov{D}$ denote the sub-$A$-module
of $\r{Der}(A)$ generated by $D$. Note that $\ov{D}$ is also a sub-Lie-algebra of $\r{Der}(A)$ and any $\beta\in\ov{D}$ is of the form
\begin{equation}\label{2111232048}
\beta=\sum_{i=1}^na_i\alpha_i\hspace{10mm}(a_i\in A,\alpha_i\in D).\end{equation} A \emph{morphism} $\phi:\f{G}'\to\f{G}$
between geometries is given by an algebra morphism $\phi:A\to A'$ such that for every $\alpha'\in\ov{D'}$ there exists $\alpha\in\ov{D}$
with the property
\begin{equation}\label{2112090807}\alpha'(\phi(a))=\phi(\alpha(a))\hspace{10mm}(a\in A).\end{equation}
A \emph{weak morphism} $\phi:\f{G}'\to\f{G}$ is an algebra morphism $\phi:A\to A'$ such that for every $\alpha'\in D'$ there is $\alpha\in D$
satisfying (\ref{2112090807}). Note that any weak morphism $\f{G}'\to\f{G}$ which is surjective as the algebra morphism $A\to A'$, is a
morphism. It is easily seen that compositions of (weak) morphisms between geometries are (weak) morphisms. Thus we have the category of geometries
with (weak) morphisms.

To any smooth manifold $X$ we may associate the
\emph{classical geometry} $$\f{X}=(\r{C}^\infty(X),\r{Vec}(X))$$ where by the abuse of notations $\r{Vec}(X)$ is also denotes the set of all
directional-derivatives $\r{d}_v:\r{C}^\infty(X)\to\r{C}^\infty(X)$ for $v\in\r{Vec}(X)$. (In the case that $X$ is not compact we have also the
another geometry $(\r{C}^\infty(X),\r{Vec}_\r{c}(X))$ associated to $X$.) Note that we have $\ov{\r{Vec}(X)}=\r{Vec}(X)$. If $f:X'\to X$ is a proper
embedding of smooth manifold $X'$ into $X$ then it follows from \cite[Problem 8-15]{Lee1} that the algebra morphism
$$\tilde{f}:\r{C}^\infty(X)\to\r{C}^\infty(X')\hspace{10mm}a\mapsto a\circ f$$ defines a morphism $\f{X}'\to\f{X}$.
Thus the category of smooth manifolds and proper embeddings may be regarded as a subcategory of the category of geometries.
\subsection{Differential forms}\label{2110300802-3}
For any geometry $\f{G}=(A,D)$ let
$$\r{d}:A\to\Lambda^1(\ov{D},A)=\r{Mod}(\ov{D},A)\hspace{5mm}(\r{d}a)(\alpha):=\alpha(a)\hspace{10mm}(a\in A,\alpha\in\ov{D}).$$
Then $\r{d}$ is a derivation called \emph{exterior-derivative}. Consider the exterior-algebra $\Lambda^*(\ov{D},A)$ and let $\Omega^*(\f{G})$
denote the subalgebra of $\Lambda^*(\ov{D},A)$ generated by $A$ and the image of $\r{d}$. Let
$$\Omega^n(\f{G}):=\Omega^*(\f{G})\cap\Lambda^n(\ov{D},A)\hspace{10mm}(n\geq0).$$ Then $\Omega^*(\f{G})=\oplus_{n=0}^\infty\Omega^n(\f{G})$
is a graded-commutative graded-algebra called the \emph{exterior-algebra} of $\f{G}$. Any element of $\Omega^n(\f{G})$ is called a
\emph{(differential) $n$-form} for $\f{G}$. Note that $\Omega^0(\f{G})=A$ and any $\omega\in\Omega^n(\f{G})$ is of the form
\begin{equation}\label{2111232101}
\omega=\sum_{j=1}^ka_0^j(\r{d}a_1^j)\wedge\cdots\wedge(\r{d}a_n^j)\hspace{10mm}(a_i^j\in A).\end{equation}
\begin{theorem}\emph{$\r{d}:\Omega^0(\f{G})\to\Omega^1(\f{G})$ extends uniquely to a graded-derivation
$$\r{d}:\Omega^*(\f{G})\to\Omega^*(\f{G})$$ of degree $1$ satisfying $\r{d}\r{d}=0$. Thus the pair $(\Omega^*(\f{G}),\r{d})$
is a differential graded-algebra.}\end{theorem}
\begin{proof}Suppose that for $a,b\in A$ we have $a(\r{d}b)=0$. Thus $a\gamma(b)=0$ for every $\gamma\in\ov{D}$.
For every $\alpha,\beta\in\ov{D}$ it follows from $\alpha(a\beta(b))=0$ that $\alpha(a)\beta(b)+a(\alpha\beta)(b)=0$. Similarly
$\beta(a)\alpha(b)+a(\beta\alpha)(b)=0$. On the other hand we have $a([\alpha,\beta](b))=0$ (since $[\alpha,\beta]\in\ov{D}$). Thus
$\alpha(a)\beta(b)-\beta(a)\alpha(b)=0$. This means that $(\r{d}a)\wedge(\r{d}b)=0$. Thus from $a(\r{d}b)=0$ we have concluded that
$(\r{d}a)\wedge(\r{d}b)=0$. Similarly it can be shown that if for a family $\{a_i^j\}_{i=0,\ldots,n}^{j=1,\ldots,k}$
of elements of $A$ we have $\sum_{j=1}^ka_0^j(\r{d}a_1^j)\wedge\cdots\wedge(\r{d}a_n^j)=0$ then
$\sum_{j=1}^k(\r{d}a_0^j)\wedge(\r{d}a_1^j)\wedge\cdots\wedge(\r{d}a_n^j)=0$. Thus the assignment
$$a_0(\r{d}a_1)\wedge\cdots\wedge(\r{d}a_n)\mapsto(\r{d}a_0)\wedge(\r{d}a_1)\wedge\cdots\wedge(\r{d}a_n)$$ defines a linear map
$\r{d}:\Omega^n(\f{G})\to\Omega^{n+1}(\f{G})$. The desired properties of $\r{d}$ can be checked.\end{proof}
For the classical geometry $\f{X}$ the objects $\r{d}$ and
$\Omega^n(\f{X})$ coincide with the classical exterior-derivative and the module $\Omega^n(X)$ of differential $n$-forms on $X$. In the following
theorem we see that the classical concepts of Lie-derivative and interior-product associated to a vector field can be stated for geometries.
The proof is similar to the classical case and omitted.
\begin{theorem}\emph{For every $\alpha\in\ov{D}$ there exists a unique graded-derivation $$\r{d}_\alpha:\Omega^*(\f{G})\to\Omega^*(\f{G})$$
of degree zero such that commutes with $\r{d}$ and such that $\r{d}_\alpha(a)=\alpha(a)$ for every $a\in A$. There exists also a unique
graded-derivation $$\f{i}_\alpha:\Omega^*(\f{G})\to\Omega^*(\f{G})$$  of degree $-1$ such that for every $a\in A$, $\f{i}_\alpha(a)=0$
and $\f{i}_\alpha(\r{d}(a))=\alpha(a)$. Moreover, for any two derivations $\alpha,\beta\in\ov{D}$ we have :
$$\r{d}_\alpha=\f{i}_\alpha\circ\r{d}+\r{d}\circ\f{i}_\alpha=(\f{i}_\alpha+\r{d})^2\hspace{5mm}\text{and}
\hspace{5mm}\f{i}_{[\alpha,\beta]}=\r{d}_\alpha\circ\f{i}_\beta-\f{i}_\beta\circ\r{d}_\alpha.$$}\end{theorem}
\subsection{de Rham cohomology}\label{2110300802-4}
The cohomology-algebra of $(\Omega^*(\f{G}),\r{d})$ is denoted by $$\r{H}_\r{dR}^*(\f{G})=\oplus_{n=0}^\infty\r{H}^n_\r{dR}(\f{G})$$
and called de Rham cohomology of $\f{G}$. Note that $\r{H}_\r{dR}^*(\f{G})$ is graded-commutative. For the classical geometry $\f{X}$,
$\r{H}_\r{dR}^*(\f{X})$ coincides with the usual de Rham cohomology of $X$.
\begin{theorem}\emph{For any morphism $\phi:\f{G}'\to\f{G}$ there exists a unique algebra morphism
$\phi:\Omega^*(\f{G})\to\Omega^*(\f{G}')$ of degree $0$ (i.e. $\phi(\Omega^n(\f{G}))\subseteq\Omega^n(\f{G}')$) which extends $\phi:A\to A'$
and commutes with the exterior-derivatives (i.e. $\phi\r{d}=\r{d}\phi$). Hence this extended morphism induces an algebra morphism
$$\r{H}_\r{dR}^*(\phi):\r{H}_\r{dR}^*(\f{G})\to\r{H}_\r{dR}^*(\f{G}').$$}\end{theorem}
\begin{proof} $\phi$ extends to differential forms by the formula $$\phi(a_0(\r{d}a_1)\wedge\cdots\wedge(\r{d}a_n)):=(\phi a_0)
(\r{d}\phi a_1)\wedge\cdots\wedge(\r{d}\phi a_n).$$
The well-definiteness and desired properties of this extended $\phi$ can be checked.\end{proof}
\subsection{Geometries on sets}\label{2111180700}
Let $\s{S}$ be a set and $\f{G}=(A,D)$ a geometry. We say that \emph{$\s{S}$ has the geometry $\f{G}$} (or \emph{$\f{G}$ is a geometry on $\s{S}$})
if $A$ is an algebra of real-valued functions on $\s{S}$ with pointwise algebra operations.
(Thus the unit of $A$ is the constant function with value $1$.) We shall see in the following that
$A$ may be regarded as the \emph{algebra of smooth functions} on $\s{S}$ and $D$ as the \emph{Lie-algebra of vector fields} on $\s{S}$.

Let $\s{S}$ be a set having the geometry $\f{G}=(A,D)$. For any point $p\in\s{S}$ we let
$$\r{T}_p(\s{S}):=D/\{\alpha\in D:(\alpha a)(p)=0, \forall a\in A\}\cong\ov{D}/\{\alpha\in\ov{D}:(\alpha a)(p)=0, \forall a\in A\}$$
where $\cong$ denotes the isomorphism of vector spaces induced by the embedding $\alpha\mapsto\alpha$ from $D$ into $\ov{D}$.
We call the vector space $\r{T}_p(\s{S})$ the \emph{tangent space} to $\s{S}$ at $p$. The disjoint union set
$\r{T}(\s{S}):=\dot{\cup}_{p\in \s{S}}\r{T}_p(\s{S})$ is called the \emph{tangent bundle} of $\s{S}$. For any $\alpha\in\ov{D}$ the
\emph{geometrization} $\alpha^\dag$ of $\alpha$ is defined to be the mapping
\begin{equation*}\alpha^\dag:\s{S}\to\r{T}(\s{S})\hspace{10mm}p\mapsto[\alpha]_p\end{equation*} where $[\alpha]_p$ denotes the image of $\alpha$
in $\r{T}_p(\s{S})$. Any $\alpha^\dag$ is called a vector field on $\s{S}$. The set of all vector fields on $\s{S}$ is denoted by $\r{Vec}(\s{S})$.
For $\alpha\in D$ it is reasonable to call $\alpha^\dag$ a \emph{basic} vector field. Thus any vector field on $\s{S}$ is a combination
of the form $\sum_{i=1}^na_i\alpha_i^\dag$ where $a_i\in A$ and $\alpha_i^\dag$ is a basic vector field on $\s{S}$. Note that $\ov{D}$ is identified
with $\r{Vec}(\s{S})$ through the assignment $\alpha\mapsto\alpha^\dag$ and hence the Lie-algebra and the $A$-module
structures on $\ov{D}$ transform to the corresponding structures on $\r{Vec}(\s{S})$. We denote $\Omega^*(\f{G})$ by $\Omega^*(\s{S})$.
For any $\omega\in\Omega^1(\s{S})$ the \emph{geometrization} of $\omega$ is defined to be the
mapping with domain $\s{S}$ that assigns to any $p\in\s{S}$ the linear functional
$$\r{T}_p(\s{S})\to\bb{R}\hspace{10mm}[\alpha]_p\mapsto\big(\omega(\alpha)\big)(p).$$
Note that the above linear functional is well-defined: We know that $\omega=\sum_ib_i\r{d}a_i$ for some $a_i,b_i\in A$; thus if $[\alpha]_p=[\beta]_p$ then we have $$(\omega(\alpha)\big)(p)=\sum_ib_i(p)\big(\alpha(a_i))(p)=\sum_ib_i(p)\big(\beta(a_i))(p)=(\omega(\beta)\big)(p).$$
More generally, for any $\omega\in\Omega^n(\f{G})$ the geometrization of $\omega$ is defined to be the mapping that assigns to any $p\in \s{S}$
the well-defined alternating $n$-linear functional $$\r{T}_p(\s{S})\times\cdots\times\r{T}_p(\s{S})\to\bb{R}\hspace{10mm}
\big([\alpha_1]_p,\ldots,[\alpha_n]_p\big)\mapsto\big(\omega(\alpha_1\ot\cdots\ot\alpha_n)\big)(p).$$
Note that any $\omega\in\Omega^n(\s{S})$ is completely distinguished by its geometrization. It follows that if $D$ is finite dimensional with
$\r{dim}(D)=m$ then $\Omega^r(\s{S})=0$ for every $r>m$.

Let $\f{G}',\f{G}$  be geometries respectively on $\s{S}',\s{S}$. A mapping $f:\s{S}'\to \s{S}$ is called
\emph{algebraic} if for every $a\in A$ we have $a\circ f\in A'$. For the algebraic mapping $f$
the assignment $a\mapsto a\circ f$ defines an algebra morphism $f^\dag:A\to A'$ called \emph{algebrisation} of $f$.

An algebraic mapping $f:\s{S}'\to \s{S}$, as above, is called \emph{(weakly) differentiable} if $f^\dag$ is a (weak)
morphism from $\f{G}'$ to $\f{G}$. Such a $f$ induces the natural mapping
\begin{equation}\label{2112090821}\r{T}(\s{S}')\to\r{T}(\s{S})\hspace{10mm}[\alpha']_{p'}\mapsto[\alpha]_{f(p')}\end{equation}
where $\alpha'$ and $\alpha$ are related with each other as in (\ref{2112090807}). The bundle mapping (\ref{2112090821}) which is also fiberwise
linear may be regarded as the \emph{derivative} of $f$.

The classical geometry $\f{X}$ is obviously a geometry on the set $X$.
For any $p\in X$, $\r{T}_p({X})$ coincides with the usual tangent space to $X$ at $p$. Also for any ordinary vector field $v$ on $X$
the geometrization of $\r{d}_v$ coincides with $v$. If $f:X'\to X$ is a proper embedding of smooth manifolds then $f$
is differentiable as a mapping between sets having geometries.
\section{A General Framework for Lifted Geometry}\label{2110300803}
From now on $X$ is a fixed smooth finite dimensional manifold without boundary. We denote by $\r{Difeo}(X)$ the group of diffeomorphisms
of $X$. The flow $\bb{R}\times X\to X$ of any complete vector field $v\in\r{Vec}(X)$ is denoted by
$$(t,x)\mapsto \r{e}^{tv}(x)\hspace{10mm}(t\in\bb{R},x\in X).$$ Let $\s{U}$ be a set with a left group-action
$$\r{Difeo}(X)\times\s{U}\to\s{U}\hspace{10mm}(\theta,s)\mapsto\theta*s.$$ A subset ${\s{S}}\subseteq\s{U}$
is called \emph{almost $\r{Vec}_\r{c}(X)$-invariant} if for any $s\in {\s{S}}$ and $v\in\r{Vec}_\r{c}(X)$ there exists
$\epsilon>0$ such that $\r{e}^{tv}*s\in {\s{S}}$ for every $t$ with $|t|<\epsilon$. As example, for any $S\subseteq\s{U}$ the sets
$$\Big\{\theta*s:\theta\in\r{Difeo}(X),s\in S\Big\}$$
$$\Big\{(e^{t_1v_1}\cdots\r{e}^{t_nv_n})*s:n\geq1,t_1,\ldots,t_n\in\bb{R},v_1,\ldots,v_n\in\r{Vec}_\r{c}(X),s\in S\Big\}$$
are almost $\r{Vec}_\r{c}(X)$-invariant.
Let ${\s{S}}\subseteq\s{U}$ be almost $\r{Vec}_\r{c}(X)$-invariant and let $f:{\s{S}}\to\bb{R}$ be a function. The \emph{directional-derivative} $\tilde{\r{d}}_vf$ of $f$ is defined by $$(\tilde{\r{d}}_vf)(s):=\lim_{t\to0}\frac{f(\r{e}^{tv}*s)-f(s)}{t}\hspace{10mm}
(v\in\r{Vec}_\r{c}(X),s\in {\s{S}}).$$ The function $f$ is called \emph{$1$-differentiable} if for every $v\in\r{Vec}_\r{c}(X)$ and $s\in {\s{S}}$
the above limit exists. If $f$ is $1$-differentiable then it is \emph{continuous} in the following sense:
$$\lim_{t\to0}f(\r{e}^{tv}*s)=f(s)\hspace{10mm}(v\in\r{Vec}_\r{c}(X),s\in {\s{S}}).$$
The function $f$ is called \emph{$n$-differentiable} ($n\geq2$) if for every $v\in\r{Vec}_\r{c}(X)$ the function $\tilde{\r{d}}_vf$ exists
and is $(n-1)$-differentiable. $f$ is called \emph{smooth} if $f$ is $n$-differentiable for every $n$.
$f$ is called \emph{linear-derivable} if it is $1$-differentiable and the mapping $v\mapsto\tilde{\r{d}}_vf$
from $\r{Vec}_\r{c}(X)$ into the vector space of all functions on ${\s{S}}$, is linear. $f$ is called \emph{Lie-compatible} if it is
$2$-differentiable and the following identity holds:
$$(\tilde{\r{d}}_v\tilde{\r{d}}_w)(f)-(\tilde{\r{d}}_w\tilde{\r{d}}_v)(f)=\tilde{\r{d}}_{[v,w]}(f)\hspace{10mm}(v,w\in\r{Vec}_\r{c}(X)).$$
If $f$ is smooth, linear-derivable and Lie-compatible then for any $u\in\r{Vec}_\r{c}(X)$, $\tilde{\r{d}}_\r{u}f$ is smooth,
linear-derivable and Lie-compatible. The mentioned fact can be seen from the following observations ($r\in\bb{R}$):
\begin{equation*}\begin{split}\tilde{\r{d}}_{v+rw}(\tilde{\r{d}}_uf)&=\tilde{\r{d}}_u(\tilde{\r{d}}_{v+rw}f)+\tilde{\r{d}}_{[v+rw,u]}(f)\\
&=\tilde{\r{d}}_u(\tilde{\r{d}}_{v}f+r\tilde{\r{d}}_{w}f)+\tilde{\r{d}}_{[v,u]}(f)+r\tilde{\r{d}}_{[w,u]}(f)\\
&=(\tilde{\r{d}}_u\tilde{\r{d}}_{v})(f)+r(\tilde{\r{d}}_u\tilde{\r{d}}_{w})(f)+\tilde{\r{d}}_{[v,u]}(f)+r\tilde{\r{d}}_{[w,u]}(f)\\
&=\tilde{\r{d}}_v(\tilde{\r{d}}_{u}f)+\tilde{\r{d}}_{[u,v]}(f)+r\tilde{\r{d}}_w(\tilde{\r{d}}_{u}f)+r\tilde{\r{d}}_{[u,w]}(f)+
\tilde{\r{d}}_{[v,u]}(f)+r\tilde{\r{d}}_{[w,u]}(f)\\
&=\tilde{\r{d}}_v(\tilde{\r{d}}_{u}f)+r\tilde{\r{d}}_w(\tilde{\r{d}}_{u}f).\end{split}\end{equation*}
\begin{equation*}\begin{split}(\tilde{\r{d}}_v\tilde{\r{d}}_w)(\tilde{\r{d}}_uf)-(\tilde{\r{d}}_w\tilde{\r{d}}_v)(\tilde{\r{d}}_uf)
=&(\tilde{\r{d}}_v\tilde{\r{d}}_u\tilde{\r{d}}_w)(f)+(\tilde{\r{d}}_v\tilde{\r{d}}_{[w,u]})(f)\\
&-(\tilde{\r{d}}_w\tilde{\r{d}}_u\tilde{\r{d}}_v)(f)-(\tilde{\r{d}}_w\tilde{\r{d}}_{[v,u]})(f)\\
=&(\tilde{\r{d}}_u\tilde{\r{d}}_v\tilde{\r{d}}_w)(f)+(\tilde{\r{d}}_{[v,u]}\tilde{\r{d}}_{w})(f)+(\tilde{\r{d}}_v\tilde{\r{d}}_{[w,u]})(f)\\
&-(\tilde{\r{d}}_u\tilde{\r{d}}_w\tilde{\r{d}}_v)(f)-(\tilde{\r{d}}_{[w,u]}\tilde{\r{d}}_v)(f)-(\tilde{\r{d}}_w\tilde{\r{d}}_{[v,u]})(f)\\
=&(\tilde{\r{d}}_u\tilde{\r{d}}_{[v,w]})(f)+\tilde{\r{d}}_{[[v,u],w]}(f)+\tilde{\r{d}}_{[v[w,u]]}(f)\\
=&(\tilde{\r{d}}_{[v,w]}\tilde{\r{d}}_u)(f)+\tilde{\r{d}}_{[u,[v,w]]}(f)+\tilde{\r{d}}_{[[v,u],w]}(f)+\tilde{\r{d}}_{[v[w,u]]}(f)\\
=&\tilde{\r{d}}_{[v,w]}(\tilde{\r{d}}_uf).\end{split}\end{equation*}
Let $A$ be an algebra of real functions on ${\s{S}}$ with pointwise operations such that:
\begin{enumerate}\item[(C1)] any function $a$ in $A$, is smooth, linear-derivable and Lie-compatible, and
\item[(C2)] every of its directional-derivatives $\tilde{\r{d}}_va$ belongs to $A$.\end{enumerate}
Then for every $v\in\r{Vec}_\r{c}(X)$, the mapping $\tilde{\r{d}}_v:A\to A$ is a derivation and the set
\begin{equation}\label{2110240850}D=\Big\{\tilde{\r{d}}_v:v\in\r{Vec}_\r{c}(X)\Big\}\end{equation} is a sub-Lie-algebra of $\r{Der}(A)$.
Thus $\f{G}=(A,D)$ is a geometry on ${\s{S}}$. As it is clear from the definition of $\f{G}$ and the results of $\S$\ref{2111180700}
the tangent vectors to $\f{G}$ are obtained as a natural \emph{lifting} of vector fields on $X$. We call the geometry $\f{G}$ a
\emph{lifted geometry} on ${\s{S}}$ if the functions in $A$ are obtained
via a \emph{uniform} and \emph{distinguished} lifting procedure of smooth functions, differential forms,
or other smooth objects associated with $X$. (The meaning of the preceding sentence
will become more clear by the examples given in the following sections.) Then once we have a distinguished way to produce smooth functions on
${\s{S}}$, by (\ref{2111232048}) and  (\ref{2111232101}) we can also produce all vector fields and differential
forms on ${\s{S}}$. Thus the basic elements of any lifted geometry on ${\s{S}}$
have two properties: (i) They can be explicitly obtained from the basic elements of the geometry of $X$. (ii) To define them there
is no need to any local coordinate system or even topology on ${\s{S}}$.

It can be easily checked that if $f,g:{\s{S}}\to\bb{R}$ are smooth ($\r{resp.}$ linear-derivable, Lie-compatible) then the functions
$rf$ ($r\in\bb{R}$), $f+g$, and $fg$ are also smooth ($\r{resp.}$ linear-derivable, Lie-compatible).
This fact together with the above results imply that
the set $B$ of all smooth, linear-derivable and Lie-compatible functions on ${\s{S}}$ is an algebra satisfying (C1) and (C2). Thus for any
lifted geometry $\f{G}$ as above we have $A\subseteq B$. But note that the geometry $(B,D)$ on ${\s{S}}$ (with $D$ as in (\ref{2110240850}))
in general can not be considered as a lifted geometry on ${\s{S}}$ because we have no control on the nature of the functions in $B$.
If $\f{G}=(A,D)$ is a lifted geometry on ${\s{S}}$ and ${\s{S}}'\subset{\s{S}}$ is almost $\r{Vec}_\r{c}(X)$-invariant then
$\f{G}|_{{\s{S}}'}:=(A|_{{\s{S}}'},D)$ is a lifted geometry on ${\s{S}}'$ where $$A|_{{\s{S}}'}:=\Big\{a|_{{\s{S}}'}:a\in A\Big\}$$ and where $D$
similar to the above is the set of derivations of the form $\tilde{\r{d}}_v$ on ${A|_{{\s{S}}'}}$ for $v\in\r{Vec}_\r{c}(X)$. Then also the inclusion
${\s{S}}'\hookrightarrow {\s{S}}$ is differentiable. We may call $\f{G}|_{{\s{S}}'}$ a \emph{restricted} lifted geometry on ${\s{S}}'$.
\begin{remark}\label{2112060732}
\emph{All the above definitions and materials and almost all the results in $\S$\ref{2110300804}-\ref{2110300808} (with some appropriate changes)
remaind valid when $\r{Vec}_\r{c}(X)$ is replaced by an arbitrary Lie-algebra $\c{L}$ of complete vector fields on $X$. Thus we may consider
the notion of almost $\c{L}$-invariant subset ${\s{S}}\subseteq\s{U}$ and the notions of $\c{L}$-smooth, linear-$\c{L}$-derivable,
and Lie-$\c{L}$-compatible functions $f:{\s{S}}\to\bb{R}$. Accordingly, we may define a lifted $\c{L}$-geometry on ${\s{S}}$ to be a geometry
$(A,D_\c{L})$ where $A$ is an algebra of functions on ${\s{S}}$ obtained via a lifting procedure and satisfying the analogues
of (C1) and (C2), and where $D_\c{L}=\{\tilde{\r{d}}_v:v\in\c{L}\}$.}\end{remark}
\section{Lifted Geometry of Spaces of Radon Measures}\label{2110300804}
In this section we extend some aspects of differential geometry for configuration spaces considered in \cite{Albeverio1} and other papers.
Let $\s{M}_{X}$ denote the cone of positive Radon measures on $X$. For any $\mu\in\s{M}_{X}$ and every $\theta\in\r{Difeo}(X)$
let $\theta_*\mu$ denote the push-forward measure of $\mu$ under $\theta$ $\text{i.e.}$
$\theta_*\mu(U):=\mu(\theta^{-1}U)$ for any Borel subset $U$ of $X$.
Thus we have the group-action $$\r{Difeo}(X)\times\s{M}_{X}\to\s{M}_{X}\hspace{10mm}(\theta,\mu)\mapsto\theta_*\mu.$$
We are going to define a lifted geometry on $\s{M}_{X}$. Then also as we saw in $\S$\ref{2110300803} any almost $\r{Vec}_\r{c}(X)$-invariant subset
of $\s{M}_{X}$ has the restricted lifted geometry. Suppose $\phi_i\in\r{C}_\r{c}^\infty(X)$ for $i=1,\ldots,n$ and $\psi\in\r{C}^\infty(\bb{R}^n)$.
We let $$F=F[\psi:\phi_1,\ldots,\phi_n]\hspace{10mm}F:\s{M}_{X}\to\bb{R}$$ be defined by
\begin{equation}\label{2111130703}F(\mu):=\psi\big(\int_X\phi_1\r{d}\mu,\ldots,\int_X\phi_n\r{d}\mu\big).\end{equation}
For any $v\in\r{Vec}_\r{c}(X)$ we may compute $\tilde{\r{d}}_vF$ as follows: For a small $\epsilon>0$ let the function
$$\varphi=(\varphi_1,\ldots,\varphi_n)\hspace{10mm}\varphi:(-\epsilon,+\epsilon)\to\bb{R}^n$$ be defined by
$$\varphi(t):=(\int_X\phi_1\r{d}(\r{e}^{tv}_*\mu),\ldots,\int_X\phi_n\r{d}(\r{e}^{tv}_*\mu)\big)=
(\int_X(\phi_1\circ\r{e}^{tv})\r{d}\mu,\ldots,\int_X(\phi_n\circ\r{e}^{tv})\r{d}\mu\big).$$
We have $$(\tilde{\r{d}}_vF)(\mu)=(\psi\circ\varphi)'(0)=\sum_{i=1}^n\varphi_i'(0)\frac{\partial\psi}{\partial r_i}\big(\varphi(0)\big)
\hspace{5mm}\text{and}\hspace{5mm}\varphi_i'(0)=\int_X(\r{d}_v\phi_i)\r{d}\mu.$$ Thus if $\xi:\bb{R}^{2n}\to\bb{R}$ is defined by
\begin{equation}\label{2110300730}
(r_1,\ldots,r_n,s_1,\ldots,s_n)\mapsto\sum_{i=1}^ns_i\frac{\partial\psi}{\partial r_i}(r_1,\ldots,r_n),\end{equation} then we have
\begin{equation}\label{2110240841}
\tilde{\r{d}}_v\big(F[\psi:\phi_1,\ldots,\phi_n]\big)=F[\xi:\phi_1,\ldots,\phi_n,\r{d}_v\phi_1,\ldots,\r{d}_v\phi_n].\end{equation}
Applying (\ref{2110240841}) two times, for $F$ as above and $v,w\in\r{Vec}_\r{c}(X)$ we have
\begin{equation*}\begin{split}\Big[\big(\tilde{\r{d}}_w\tilde{\r{d}}_v\big)F\Big]\big(\mu\big)=&\hspace{2.5mm}
\sum_{j,i=1}^n\Big(\int_X(\r{d}_w\phi_j)\r{d}\mu\Big)\Big(\int_X(\r{d}_v\phi_i)\r{d}\mu\Big)
\Big[\frac{\partial\psi}{\partial r_j\partial r_i}\big(\int_X\phi_1\r{d}\mu,\ldots,\int_X\phi_n\r{d}\mu\big)\Big]\\
&+\sum_{k=1}^n\Big(\int_X(\r{d}_w\r{d}_v\phi_k)\r{d}\mu\Big)
\Big[\frac{\partial\psi}{\partial r_k}\big(\int_X\phi_1\r{d}\mu,\ldots,\int_X\phi_n\r{d}\mu\big)\Big].\end{split}\end{equation*}
Similarly $\big[(\tilde{\r{d}}_v\tilde{\r{d}}_w)F\big](\mu)$ may be computed explicitly, and then we find out that
\begin{equation*}\begin{split}
&\Big[\big(\tilde{\r{d}}_v\tilde{\r{d}}_w\big)F\Big]\big(\mu\big)-\Big[\big(\tilde{\r{d}}_w\tilde{\r{d}}_v\big)F\Big]\big(\mu\big)\\
=&\sum_{k=1}^n\Big(\int_X(\r{d}_v\r{d}_w\phi_k-\r{d}_w\r{d}_v\phi_k)\r{d}\mu\Big)
\Big[\frac{\partial\psi}{\partial r_k}\big(\int_X\phi_1\r{d}\mu,\ldots,\int_X\phi_n\r{d}\mu\big)\Big]\\
=&\sum_{k=1}^n\Big(\int_X(\r{d}_{[v,w]}\phi_k)\r{d}\mu\Big)
\Big[\frac{\partial\psi}{\partial r_k}\big(\int_X\phi_1\r{d}\mu,\ldots,\int_X\phi_n\r{d}\mu\big)\Big]\\
=&\Big[\tilde{\r{d}}_{[v,w]}F\Big]\big(\mu\big).\end{split}\end{equation*}
Thus we have showed that any function $F[\psi:\phi_1,\ldots,\phi_n]$ is Lie-compatible. We have
\begin{equation}\label{2110300712}\begin{split}
&F[\psi\bar{+}\psi':\phi_1,\ldots,\phi_n,\phi_1',\ldots,\phi'_{n'}]=F[\psi:\phi_1,\ldots,\phi_n]+F[\psi':\phi_1',\ldots,\phi'_{n'}]\\
&F[\psi\bar{\times}\psi':\phi_1,\ldots,\phi_n,\phi_1',\ldots,\phi'_{n'}]=\big(F[\psi:\phi_1,\ldots,\phi_n]\big)
\big(F[\psi':\phi_1',\ldots,\phi'_{n'}]\big)\end{split}\end{equation} where $\psi\bar{+}\psi',\psi\bar{\times}\psi'\in\r{C}^\infty(\bb{R}^{n+n'})$
are given respectively by $(r_1,\ldots,r_{n+n'})\mapsto$ $$\psi(r_1,\ldots,r_n)+\psi'(r_{n+1},\ldots,r_{n+n'})\hspace{2mm}\text{and}\hspace{2mm}
\psi(r_1,\ldots,r_n)\psi'(r_{n+1},\ldots,r_{n+n'}).$$ Let
$$A:=\Big\{F[\psi:\phi_1,\ldots,\phi_n]:\psi\in\r{C}^\infty(\bb{R}^n),\phi_1,\ldots,\phi_n\in\r{C}^\infty_\r{c}(X),n\geq1\Big\}.$$
It is concluded from (\ref{2110300712}) that $A$ is an algebra of functions on $\s{M}_{X}$. Also it follows from the formula
(\ref{2110240841}) that the functions in $A$ are smooth and linear-derivable. Thus the conditions (C1) and (C2) for $A$
are satisfied and we have the geometry $(A,D)$ on $\s{M}_{X}$ where $D$ is given by (\ref{2110240850}). The functions in
$A$ are obtained via the uniform and distinguished lifting procedure, given by the formula (\ref{2111130703}), of the smooth functions
on $X$. Thus $(A,D)$ may be regarded as a lifted geometry on $\s{M}_{X}$.

For any $\mu\in\s{M}_{X}$ let $\sim_\mu$ be the equivalence relation on $\r{Vec}_\r{c}(X)$ given by
$$\Big(v\sim_\mu w\Big)\Longleftrightarrow\Big(v(x)=w(x)\hspace{2mm}\text{for almost all}\hspace{1mm}x\in X\hspace{1mm}\text{w.r.t.}
\hspace{1mm}\mu\Big).$$ Then it can be checked that the assignment $v\mapsto\tilde{\r{d}}_v$ induces a vector-space isomorphism from
$\r{Vec}_\r{c}/\sim_\mu$ onto $\r{T}_\mu(\s{M}_{X})$.

The restricted lifted geometry on the following $\r{Difeo}(X)$-invariant subsets of $\s{M}_X$ could be considered:
(i) The subset $\s{M}^\r{f}_X$ of finite measures. (ii) The subset of measures $\mu$ with $\mu(X)\leq r$ for some
fixed number $r$. (iii) The subset of measures with values in $\bb{N}$. (iv) The configuration space of $X$ \cite{Albeverio0,Albeverio1,Albeverio2},
that is the subset of measures $\mu$ of the form $\sum_{x\in K}\delta_x$ where $K$ is a subset of $X$ without any limit point.
(Thus $K$ is countable.) (v) The subset of probability measures on $X$. (vi) The subset of measures on $X$ induced by Riemannian metrics on
$k$-dimensional submanifolds of $X$ for some fixed $k\leq\r{dim}(X)$. (vii) The subset $\s{M}^\r{c}_X$
of measures with compact supports. (viii) The subset of measures without any atom. (ix) The set of Radon measures which are absolutely continuous
$\text{w.r.t.}$ a measure induced by a Riemannian metric on $X$.

Let $\Upsilon:X'\to X$ be a proper embedding of a smooth manifold $X'$ into $X$. Consider the induced mapping
$\hat{\Upsilon}:\s{M}_{X'}\to\s{M}_{X}$ given by $\mu'\mapsto\Upsilon_*\mu'$. We have
$$\big(F[\psi:\phi_1,\ldots,\phi_n]\big)\circ\hat{\Upsilon}=F[\psi:\phi_1\circ\Upsilon,\ldots,\phi_n\circ\Upsilon].$$
Thus $\hat{\Upsilon}$ is algebraic. Also it follows from \cite[Lemma 5.34]{Lee1} that its algebrisation
is a surjective algebra morphism. We know that for every $v'\in\r{Vec}_\r{c}(X')$ there is $v\in\r{Vec}_\r{c}(X)$
that extends $v'$ i.e. $v'=v\circ\Upsilon$ where $\r{T}(X')$ is identified with a subset of $\r{T}(X)$. We have
\begin{equation*}\begin{split}
\tilde{\r{d}}_{v'}((F[\psi:\phi_1,\ldots,\phi_n])\circ\hat{\Upsilon})&=\tilde{\r{d}}_{v'}(F[\psi:\phi_1\circ\Upsilon,\ldots,\phi_n\circ\Upsilon])\\
&=F[\xi:\phi_1\circ\Upsilon,\ldots,\phi_n\circ\Upsilon,\r{d}_{v'}(\phi_1\circ\Upsilon),\ldots,\r{d}_{v'}(\phi_n\circ\Upsilon)]\\
&=F[\xi:\phi_1\circ\Upsilon,\ldots,\phi_n\circ\Upsilon,(\r{d}_{v}\phi_1)\circ\Upsilon,\ldots,(\r{d}_{v}\phi_n)\circ\Upsilon]\\
&=(F[\xi:\phi_1,\ldots,\phi_n,\r{d}_{v}\phi_1,\ldots,\r{d}_{v}\phi_n])\circ\hat{\Upsilon}\\
&=(\tilde{\r{d}}_vF[\psi:\phi_1,\ldots,\phi_n])\circ\hat{\Upsilon}.\end{split}\end{equation*}
Thus $\hat{\Upsilon}$ is differentiable.
We may regard $X\mapsto\s{M}_{X}$ as a functor from the category of manifolds and proper embeddings to the category of sets having geometries.

Suppose that $X$ has a Lie-group structure. For any measure $\nu\in\s{M}^\r{c}_{X})$ consider the convolution-mapping $\hat{\nu}$
given by $\mu\mapsto\mu\star\nu$ from $\s{M}^\r{f}_{X}$ into itself. We have $$\big(F[\psi:\phi_1,\ldots,\phi_n]\big)\circ\hat{\nu}=
F[\psi:(x\mapsto\int_X\phi_1(xy)\r{d}\nu(y)),\ldots,(x\mapsto\int_X\phi_n(xy)\r{d}\nu(y))].$$ Thus $\hat{\nu}$ is algebraic.
Similarly, $\mu\mapsto\nu\star\mu$ is algebraic.

For any $f\in\r{C}^\infty(X)$ let $\hat{f}:\s{M}_{X}\to\s{M}_{X}$ be defined by $\r{d}(\hat{f}\mu):=f\r{d}\mu$. The following identity shows that
$\hat{f}$ is algebraic: $$\big(F[\psi:\phi_1,\ldots,\phi_n]\big)\circ\hat{f}=F[\psi:f\phi_1,\ldots,f\phi_n].$$
\section{Gradient in Lifted Riemannian Geometry}\label{2110300804.5}
In this section we extend some contents considered in \cite{Albeverio4,Albeverio0,Albeverio1,Kondratiev1,Kuchling1,Ma1,Rockner1}.
Suppose that $X$ has a Riemannian metric $g$ and let $\s{M}_X$ and $A$ be as in $\S$\ref{2110300804}.
Let $\s{M}\subseteq\s{M}_X$ be an almost $\r{Vec}_\r{c}(X)$-invariant subset. For $\mu\in\s{M}_X$, we have the well-defined inner product
$$\sl[v]_\mu,[w]_\mu\sr_g:=\int_X\sl v(x),w(x)\sr_g\r{d}\mu(x)$$ on $\r{T}_\mu(\s{M})$. Hence we may regard $\s{M}$ as a Riemannian manifold.
For $F\in A$ the \emph{gradient} $\nabla F$ of $F$ is a vector field
on $\s{M}$ satisfying $$\big(\r{d}F(\mu)\big)[v]_\mu=\sl [v]_\mu,\nabla F(\mu)\sr_g\hspace{10mm}(\mu\in\s{M},v\in\r{Vec}_\r{c}(X)).$$
We prove that $\nabla F$ is actually a member of $\r{Vec}(\s{M})$:
Suppose that the function $F\in A$ be given by (\ref{2111130703}). For any fixed $\mu\in\s{M}$ we show that there is a
canonical vector field $w^\mu\in\r{Vec}_\r{c}(X)$ satisfying
\begin{equation}\label{2112170702}\big(\r{d}F(\mu)\big)[v]_\mu=\sl [v]_\mu,[w^\mu]_\mu\sr_g\hspace{10mm}(v\in\r{Vec}_\r{c}(X))\end{equation}
The more explicit form of equation (\ref{2112170702}) for every $v\in\r{Vec}_\r{c}(X)$ is
\begin{equation}\label{2112170703}\sum_{i=1}^nF_i(\mu)\int_{X}\big(\r{d}_v\phi_i\big)\r{d}\mu=\int_X\sl v(x),w^\mu(x)\sr_g\r{d}\mu(x)\end{equation}
where $F_i$ denotes the function $F[\frac{\partial\psi}{\partial r_i}:\phi_1,\ldots,\phi_n]$ in $A$.
Let $\c{O}$ be any open subset of $X$ which is identified with $\bb{R}^m$ ($\r{dim}(X)=m$) via a local coordinate mapping. Using the
identification $\c{O}\cong\bb{R}^m$ we may regard the \emph{restriction} of any object appearing in (\ref{2112170703}) as the corresponding object on
$\bb{R}^m$. Then for any $v\in\r{Vec}_\r{c}(X)$ with $\r{Supp}(v)\subset\c{O}$, (\ref{2112170703}) becomes
\begin{equation}\label{2112170704}\sum_{i=1}^nF_i(\mu)\int_{\bb{R}^m}\sum_{j=1}^mv_j\frac{\partial\phi_i}{\partial x_j}\r{d}\mu
=\int_{\bb{R}^m}\big(\sum_{j,k=1}^mv_jw^\mu_kg_{jk}\big)\r{d}\mu.\end{equation} It is important to note that since $\r{Supp}(v)\subset\c{O}$
the left (resp. right) hand sides of (\ref{2112170703}) and (\ref{2112170704}) are equal. Rearranging the sums in (\ref{2112170704}) we get
\begin{equation*}\sum_{j=1}^m\int_{\bb{R}^m}v_j\Big(\sum_{i=1}^nF_i(\mu)\frac{\partial\phi_i}{\partial x_j}\Big)\r{d}\mu=
\sum_{j=1}^m\int_{\bb{R}^m}v_j\Big(\sum_{k=1}^mw^\mu_kg_{jk}\Big)\r{d}\mu.\end{equation*}
Since $v$ is arbitrary it is concluded that for every $j=1,\ldots,m$ we must have
$$\sum_{i=1}^nF_i(\mu)\frac{\partial\phi_i}{\partial x_j}=\sum_{k=1}^mw^\mu_kg_{jk}\hspace{4mm}\text{almost every where w.r.t.}\hspace{1mm}\mu,$$
Hence for every $k=1,\ldots,m$ we must have
\begin{equation}\label{2112170705}
w^\mu_k=\sum_{j=1}^m\sum_{i=1}^ng^{-1}_{kj}F_i(\mu)\frac{\partial\phi_i}{\partial x_j}\hspace{4mm}\text{almost every where w.r.t.}\hspace{1mm}\mu.
\end{equation} For every $i=1,\ldots,n$ let $u^i=\nabla\phi_i$ denote the gradient of $\phi_i$ $\text{w.r.t.}$ $g$ on $X$.
$u^i\in\r{Vec}_\r{c}(X)$ and its components in a local coordinate system $\c{O}\cong\bb{R}^m$ as above is given by
\begin{equation}\label{2112170706}u^i_k:=\sum_{j=1}^mg^{-1}_{kj}\frac{\partial\phi_i}{\partial x_j}\hspace{10mm}(k=1,\ldots,m).\end{equation}
Now it is concluded from (\ref{2112170705}) and (\ref{2112170706}) that if we let $w^\mu$ to be defined by
$$w^\mu:=\sum_{i=1}^nF_i(\mu)u^i\in\r{Vec}_\r{c}(X)$$ then it satisfies in equation (\ref{2112170702}). Then also it is clear that
$$\nabla F=\sum_{i=1}^nF_i(u^i)^\dag\in\r{Vec}(\s{M}).$$ More explicitly we have
$$\nabla F[\psi:\phi_1,\ldots,\phi_n](\mu)=\sum_{i=1}^nF
\big[\frac{\partial\psi}{\partial r_i}:\phi_1,\ldots,\phi_n\big](\mu)\big[\nabla\phi_i\big]_\mu.$$
We endow $\s{M}$ with the \emph{weak topology} that is defined to be the smallest topology under which every function
$\s{M}\ni\mu\mapsto\int_X\phi\r{d}(\mu)$ for $\phi\in\r{C}^\infty_\r{c}(X)$ is continuous. It can be checked that the weak topology is Hausdorff.
Any function $F\in A$ may be regarded as a continuous function on $\s{M}$. Let $\Theta$ be a Borel probability measure on $\s{M}$.
Thus $\Theta$ may be regarded as a \emph{random radon measure} on $X$.
We are going to consider a construction of the \emph{formal} Laplace operator for $\s{M}$ $\text{w.r.t.}$ the pair $(g,\Theta)$, by means of its
associated quadratic form $\f{L}$ on $\r{L}^2(\Theta)$. Let $A_\r{c}\subset A$ be the subset of those functions of the form (\ref{2111130703})
with $\psi\in\r{C}^\infty_\r{c}(\bb{R}^n)$. Then $A_\r{c}$ is a subalgebra of bounded continuous functions on $\s{M}$ and hence
$A_\r{c}\subset\r{L}^2(\Theta)$. If $\s{M}$ is compact ($\text{e.g.}$ $X$ is compact and $\s{M}$ is the set of probability measures)
then $A_\r{c}$ is also dense in $\r{L}^2(\Theta)$. We let the symmetric positive-definite bilinear functional
$$\f{L}:A_\r{c}\times A_\r{c}\to\bb{R},$$ which may be called \emph{Dirichlet form} associated with $\Theta$,
be defined by $$\f{L}(F,F'):=\int_\s{M}\sl\nabla F(\mu),\nabla F'(\mu)\sr_g\r{d}\Theta(\mu).$$
More explicitly for $F=F[\psi,\phi_1,\ldots,\phi_n]$ and $F'=F[\psi',\phi'_1,\ldots,\phi'_{n'}]$ in $A_\r{c}$, $\f{L}(F,F')$ is the integral
of the following function of $\mu$ on $\s{M}$ $\text{w.r.t.}$ $\Theta$:
\begin{equation}\label{2112200744}\sum_{i=1}^n\sum_{i'=1}^{n'}F[\frac{\partial\psi}{\partial r_i}:\phi_1,\ldots,\phi_n](\mu)
F[\frac{\partial\psi'}{\partial r_{i'}}:\phi'_1,\ldots,\phi'_{n'}](\mu)\int_X\sl\nabla\phi_i,\nabla\phi'_{i'}\sr_g\r{d}\mu\end{equation}
Suppose that $\epsilon\in\r{C}^\infty(\bb{R})$. Then we have
\begin{equation}\label{2112200745}\epsilon\circ F[\psi,\phi_1,\ldots,\phi_n]=F[\epsilon\circ\psi,\phi_1,\ldots,\phi_n],\end{equation}
\begin{equation}\label{2112200746}F[\frac{\partial(\epsilon\circ\psi)}{\partial r_i}:\phi_1,\ldots,\phi_n](\mu)
=\Big(\frac{\partial\epsilon}{\partial t}\big(F[\psi,\phi_1,\ldots,\phi_n](\mu)\big)\Big)
F[\frac{\partial\psi}{\partial r_i}:\phi_1,\ldots,\phi_n](\mu).\end{equation}
Let $F=F[\psi,\phi_1,\ldots,\phi_n]$ be in $A_\r{c}$. Suppose that $\epsilon$ as above has the properties $$\epsilon(0)=0\hspace{4mm}\text{and}
\hspace{4mm}-1\leq\frac{\partial\epsilon}{\partial t}\leq1.$$ Since $\epsilon\circ\psi\in\r{C}^\infty_\r{c}(\bb{R}^n)$, (\ref{2112200745})
implies that $\epsilon\circ F\in A_\r{c}$. (\ref{2112200746}) shows that $\f{L}(\epsilon\circ F,\epsilon\circ F)$ is the integral
$\text{w.r.t.}$ $\Theta$ of the positive function given by (\ref{2112200744}) with $n'=n,\psi'=\psi,\phi'_i=\phi_i$, multiplied by the function
$$\mu\mapsto\Big(\frac{\partial\epsilon}{\partial t}\big(F(\mu)\big)\Big)^2.$$ Thus we have
$$\f{L}(\epsilon\circ F,\epsilon\circ F)\leq\f{L}(F,F)\hspace{10mm}(F\in A_\r{c}).$$ This implies that $\f{L}$ is a \emph{Markovian} form in
the sense of \cite{Fukushima1}.
\section{Lifted Geometry of Mapping Spaces}\label{2110300805}
Let $Y$ be a set with a $\sigma$-algebra $\Sigma$ of its subsets. Denote by $\s{F}^Y_X$ the set of all Borel measurable mappings from
$Y$ into $X$. We have the canonical group-action
$$\r{Difeo}(X)\times\s{F}^Y_X\to\s{F}^Y_X\hspace{10mm}(\theta,P)\mapsto\theta\circ P.$$
We are going to define a class of lifted geometries for $\s{F}^Y_X$.  For any $n$-tuple $(\mu_1,\ldots,\mu_n)$ of finite positive
measures on $(Y,\Sigma)$ and any function $\phi\in\r{C}^\infty_\r{c}(X^n)$ let the function $$F=F[\phi:\mu_1,\ldots,\mu_n]\hspace{10mm}
F:\s{F}^Y_X\to\bb{R}$$ be defined by $$F(P):=\int_{Y^n}\phi(P,\ldots,P)\r{d}(\mu_1\times\cdots\times\mu_n)\hspace{10mm}(P\in\s{F}^Y_X).$$
For any $v\in\r{Vec}_\r{c}(X)$ we have
\begin{equation*}\begin{split}(\tilde{\r{d}}_vF)(P)&=\lim_{t\to0}\frac{1}{t}
\int_{Y^n}\Big[\phi\big(\r{e}^{tv}(Py_1),\ldots,\r{e}^{tv}(Py_n)\big)-\phi\big(Py_1,\ldots,Py_n\big)\Big]\r{d}\mu_1(y_1)\cdots\r{d}\mu_n(y_n)\\
&=\int_{Y^n}\big(\r{d}_{v^{\oplus n}}\phi\big)\big(P,\ldots,P\big)\r{d}(\mu_1\times\cdots\times\mu_n)\end{split}\end{equation*}
where $v^{\oplus n}\in\r{Vec}_\r{c}(X^n)$ denotes the direct sum of $n$ copies of $v$. Thus we have
\begin{equation}\label{2110280706}\tilde{\r{d}}_v\big(F[\phi:\mu_1,\ldots,\mu_n]\big)=F[\r{d}_{v^{\oplus n}}\phi:\mu_1,\ldots,\mu_n].\end{equation}
For $v,w\in\r{Vec}_\r{c}(X)$ and $F$ as above by applying (\ref{2110280706}) we have that
\begin{equation}\label{2110280709}\begin{split}&(\tilde{\r{d}}_v\tilde{\r{d}}_wF)(P)-(\tilde{\r{d}}_w\tilde{\r{d}}_vF)(P)\\
=&\int_{Y^n}\big[(\r{d}_{v^{\oplus n}}\r{d}_{w^{\oplus n}}\phi)(P,\ldots,P)-(\r{d}_{w^{\oplus n}}\r{d}_{v^{\oplus n}}\phi)(P,\ldots,P)\big]
\r{d}(\mu_1\times\cdots\times\mu_n)\\=&\int_{Y^n}(\r{d}_{v^{\oplus n}}\r{d}_{w^{\oplus n}}\phi-\r{d}_{w^{\oplus n}}
\r{d}_{v^{\oplus n}}\phi)(P,\ldots,P)\r{d}(\mu_1\times\cdots\times\mu_n)\\=&\int_{Y^n}(\r{d}_{[v,w]^{\oplus n}}\phi)
(P,\ldots,P)\r{d}(\mu_1\times\cdots\times\mu_n)\\=&(\tilde{\r{d}}_{[v,w]}F)(P).\end{split}\end{equation}We have also the identities
\begin{equation}\label{2110280712}\begin{split}
&F[\phi\bar{+}\phi':\mu_1,\ldots,\mu_n,\mu_1',\ldots,\mu'_{n'}]=F[\phi:\mu_1,\ldots,\mu_n]+F[\phi':\mu_1',\ldots,\mu'_{n'}]\\
&F[\phi\bar{\times}\phi':\mu_1,\ldots,\mu_n,\mu_1',\ldots,\mu'_{n'}]=\big(F[\phi:\mu_1,\ldots,\mu_n]\big)\big(F[\phi':\mu_1',\ldots,\mu'_{n'}]\big)
\end{split}\end{equation} where $\phi\bar{+}\phi',\phi\bar{\times}\phi'\in\r{C}^\infty_\r{c}(X^{n+n'})$ are given respectively by
$(x_1,\ldots,x_{n+n'})\mapsto$ $$\phi(x_1,\ldots,x_n)+\phi'(x_{n+1},\ldots,x_{n+n'})\hspace{2mm}\text{and}\hspace{2mm}
\phi(x_1,\ldots,x_n)\phi'(x_{n+1},\ldots,x_{n+n'})$$ Let $\c{Y}$ be any nonempty family of finite positive measures on $(Y,\Sigma)$. Let
$$A:=\Big\{F[\phi:\mu_1,\ldots,\mu_n]:\phi\in\r{C}^\infty_\r{c}(X^n),\mu_1,\ldots,\mu_n\in\c{Y},n\in\bb{N}\Big\}.$$
It follows from (\ref{2110280706})-(\ref{2110280712}) that $A$ is an algebra of functions on $\s{F}^Y_X$ satisfying (C1) and (C2).
Thus we have defined a lifted geometry $(A,D)$ for $\s{F}^Y_X$ where $D$ is given by (\ref{2110240850}).
For any $P\in\s{F}^Y_X$ and every $v,w\in\r{Vec}_\r{c}(X)$ we write $v\sim_Pw$ if  the mappings $v\circ P$ and $w\circ P$ from $Y$ into the tangent
bundle of $X$ are almost every where equal w.r.t. every $\mu\in\c{Y}$. Then $\sim_P$ is an equivalence relation on $\r{Vec}_\r{c}(X)$
and it can be checked that the assignment $v\mapsto\tilde{\r{d}}_v$ induces a surjective vector space isomorphism
$\r{Vec}_\r{c}(X)/\sim_P\to\r{T}_P(\s{F}^Y_X)$.

For any proper embedding $\Upsilon:X'\to X$ consider the induced mapping $$\hat{\Upsilon}:\s{F}^Y_{X'}\to\s{F}^Y_X
\hspace{10mm}P'\mapsto\Upsilon\circ P'.$$ We have
$$\big(F[\phi:\mu_1,\ldots,\mu_n]\big)\circ\hat{\Upsilon}=F[\phi\circ\Upsilon^{\oplus^n}:\mu_1,\ldots,\mu_n].$$
If $\phi'\in\r{C}^\infty_\r{c}({X'}^n)$ there exists $\phi\in\r{C}^\infty_\r{c}(X^n)$ such that $\phi'=\phi\circ\Upsilon^{\oplus^n}$.
Thus the algebrisation of $\hat{\Upsilon}$ is surjective. If $v\in\r{Vec}_\r{c}(X)$ extends $v'\in\r{Vec}_\r{c}(X')$ then
\begin{equation*}\begin{split}\tilde{\r{d}}_{v'}\big((F[\phi:\mu_1,\ldots,\mu_n])\circ\hat{\Upsilon}\big)&=
F[\r{d}_{{v'}^{\oplus^n}}(\phi\circ\Upsilon^{\oplus^n}):\mu_1,\ldots,\mu_n]\\
&=F[\r{d}_{{v}^{\oplus^n}}\phi:\mu_1,\ldots,\mu_n]\circ\hat{\Upsilon}\\
&=\big(\tilde{\r{d}}_v(F[\phi:\mu_1,\ldots,\mu_n])\big)\circ\hat{\Upsilon}.\end{split}\end{equation*}
Thus $\hat{\Upsilon}$ is differentiable.

Let $\pi:(Y,\Sigma)\to(Y',\Sigma')$ be a measurable mapping and $\c{Y}'$ a set of finite positive measures on $Y'$ such that $\pi_*\mu\in\c{Y}'$
for every $\mu\in\c{Y}$. Consider the induced mapping $$\hat{\pi}:\s{F}^{Y'}_{X}\to\s{F}^Y_X\hspace{10mm}P'\mapsto P'\circ\pi.$$
We have $$\big(F[\phi:\mu_1,\ldots,\mu_n]\big)\circ\hat{\pi}=F[\phi:\pi_*\mu_1,\ldots,\pi_*\mu_n],$$ and hence $\hat{\pi}$ is differentiable.

The assignments $X\mapsto\s{F}^Y_X$ and $(Y,\Sigma,\c{Y})\mapsto\s{F}^Y_X$ may be regarded as (co)functors.

In case $Y$ is a smooth manifold the restricted lifted geometry of the $\r{Difeo}(X)$-invariant set $\r{C}^\infty(Y,X)$ of all smooth
mappings from $Y$ into $X$ can be considered. The geometry of $\r{C}^\infty(Y,X)$ as an infinite dimensional manifold locally modeled on appropriate
topological vector spaces has been considered by many authors.
\section{Lifted Geometry of Spaces of Submanifolds}\label{2110300806}
Let $\s{E}_X^{k}$ denote the set of all embedded oriented submanifolds of $X$ (with or without boundary) of the fixed dimension
${k}\leq\r{dim}(X)$. We have the obvious group-action $$\r{Difeo}(X)\times\s{E}_X^{k}\to\s{E}_X^{k}\hspace{10mm}(\theta,E)\mapsto\theta(E).$$
For any $n$-tuple $(\omega_1,\ldots,\omega_n)$ of ${k}$-differential forms $\omega_i\in\Omega_\r{c}^{k}(X)$ on $X$ with compact support and
any $\psi\in\r{C}^\infty(\bb{R}^n)$ we let the function $$F=F[\psi:\omega_1,\ldots,\omega_n]\hspace{10mm}F:\s{E}_X^{k}\to\bb{R}$$ be defined by
$$F(E):=\psi(\int_E\omega_1,\ldots,\int_E\omega_n).$$ Similar to $\S$\ref{2110300804} it can be shown that for any $v\in\r{Vec}_\r{c}(X)$ we have
\begin{equation}\label{2110310701}
\tilde{\r{d}}_v\big(F[\psi:\omega_1,\ldots,\omega_n]\big)=F[\xi:\omega_1,\ldots,\omega_n,\r{d}_v\omega_1,\ldots,\r{d}_v\omega_n]
\end{equation} where $\xi:\bb{R}^{2n}\to\bb{R}$ is defined by (\ref{2110300730}). Also it can be checked that the set $$A_{k}:=
\Big\{F[\psi:\omega_1,\ldots,\omega_n]:\psi\in\r{C}^\infty(\bb{R}^n),\omega_1,\ldots,\omega_n\in\Omega^{k}_\r{c}(X),n\geq1\Big\}$$
is an algebra of functions on ${\s{E}_X^{k}}$ satisfying (C1) and (C2). Thus $(A_{k},D)$
may be regarded a lifted geometry on $\s{E}_X^{k}$ where $D$ is given by (\ref{2110240850}).
For any $E\in\s{E}_X^{k}$ and every $v,w\in\r{Vec}_\r{c}(X)$ write $v\sim_Ew$ if $v|_E=w|_E$.
Then the vector spaces $\r{Vec}_\r{c}/\sim_E$ and $\r{T}_E(\s{E}_X^{k})$ are canonically isomorphic.

For any proper embedding $\Upsilon:X'\to X$ consider the mapping $\hat{\Upsilon}:\s{E}_{X'}^{k}\to\s{E}_X^{k}$ defined by $E'\mapsto\Upsilon(E')$.
We have $$\big(F[\psi:\omega_1,\ldots,\omega_n]\big)\circ\hat{\Upsilon}=F[\psi:\Upsilon^*\omega_1,\ldots,\Upsilon^*\omega_n].$$
It is well-know that any $\omega'\in\Omega^{k}_\r{c}(X')$ extends to some $\omega\in\Omega^{k}_\r{c}(X)$ i.e. $\omega'=\Upsilon^*\omega$.
Thus the algebrisation of $\hat{\Upsilon}$ is surjective. For $v\in\r{Vec}_\r{c}(X)$ that extends $v'\in\r{Vec}_\r{c}(X')$ we have
\begin{equation*}\begin{split}\tilde{\r{d}}_{v'}\big((F[\psi:\omega_1,\ldots,\omega_n])\circ\hat{\Upsilon}\big)&=
F[\xi:\Upsilon^*\omega_1,\ldots,\Upsilon^*\omega_n,\r{d}_{v'}(\Upsilon^*\omega_1),\ldots,\r{d}_{v'}(\Upsilon^*\omega_n)]\\
&=\big(\tilde{\r{d}}_v(F[\psi:\omega_1,\ldots,\omega_n])\big)\circ\hat{\Upsilon}.\end{split}\end{equation*}
Thus $\hat{\Upsilon}$ is differentiable.

Let $\s{E}_X^{k,\r{b}}\subset\s{E}_X^{k}$ denote the subset of submanifolds with nonempty boundary. Thus $\s{E}_X^{k,\r{b}}$ has
the restricted lifted geometry induced from $\s{E}_X^{k}$. Consider the boundary operator $$\partial:\s{E}_X^{k,\r{b}}\to\s{E}_X^{k-1}$$
that associates to any $E\in\s{E}_X^{k,\r{b}}$ its boundary $\partial E$. By Stokes' Theorem we have $$\big(F[\psi:\omega_1,\ldots,\omega_n]\big)
\big(\partial E)=\big(F[\psi:\r{d}\omega_1,\ldots,\r{d}\omega_n]\big)\big(E\big)\hspace{10mm}(E\in\s{E}_X^{k,\r{b}})$$ for every
$F[\psi:\omega_1,\ldots,\omega_n]$ in $A_{k-1}(X)$. Thus $\partial$ is algebraic. For $v\in\r{Vec}_\r{c}(X)$ we have
\begin{equation*}\begin{split}\tilde{\r{d}}_{v}\big((F[\psi:\omega_1,\ldots,\omega_n])\circ\partial\big)&=
\tilde{\r{d}}_{v}\big(F[\psi:\r{d}\omega_1,\ldots,\r{d}\omega_n]\big)\\
&=F[\xi:\r{d}\omega_1,\ldots,\r{d}\omega_n,\r{d}_{v}(\r{d}\omega_1),\ldots,\r{d}_{v}(\r{d}\omega_n)]\\
&=F[\xi:\r{d}\omega_1,\ldots,\r{d}\omega_n,\r{d}(\r{d}_v\omega_1),\ldots,\r{d}(\r{d}_v\omega_n)]\\
&=(F[\xi:\omega_1,\ldots,\omega_n,\r{d}_v\omega_1,\ldots,\r{d}_v\omega_n])\circ\partial\\
&=\big(\tilde{\r{d}}_{v}(F[\psi:\omega_1,\ldots,\omega_n])\big)\circ\partial.\end{split}\end{equation*}
Thus $\partial$ is weakly differentiable.
\section{Lifted Geometry of Spaces of Tilings}\label{2110300807}
By a \emph{tiling} on $X$ we mean a (countable) set $T$ of pairwise disjoint connected open subsets of $X$ satisfying the following three conditions:
\begin{enumerate}\item[(i)] The closure $\ov{U}$ of any $U\in T$ is compact.\item[(ii)] $X=\cup_{U\in T}\ov{U}$.
\item[(iii)] For every $x\in X$ there is an open set $V_x$ containing $x$ with $V_x\cap U\neq\emptyset$ only for a finite number
of members $U$ of $T$.\end{enumerate} We denote by ${\s{T}_X}$ the set of all tilings on $X$. We have the group-action given by
$$\r{Difeo}(X)\times{\s{T}_X}\to{\s{T}_X}\hspace{10mm}\theta*T:=\Big\{\theta(U):U\in T\Big\}.$$
Suppose that $X$ is oriented and $\r{dim}(X)=\r{k}$. We consider every open subset of $X$ as
an oriented submanifold. For any $T\in{\s{T}_X}$ and any open subset $V$ of $X$ we let
$$T|V:=\cup_{U\in T,U\cap V\neq\emptyset}U.$$ It can be checked that if $\ov{V}$ is compact then
\begin{enumerate}\item[(iv)] $\ov{T|V}$ is compact, and \item[(v)] for any $v\in\r{Vec}_\r{c}(X)$ there exists $\epsilon>0$ such that for
every $t$ with $|t|<\epsilon$ we have:$$(\r{e}^{tv}*T)|V=T|V.$$\end{enumerate}
For $n$-tuples $(V_1,\ldots,V_n)$ and $(\omega_1,\ldots,\omega_n)$ of open subsets $V_i$ of $X$ with $\ov{V_i}$ compact and
$\r{k}$-differential forms $\omega_i\in\Omega^\r{k}(X)$ on $X$, and any $\psi\in\r{C}^\infty(\bb{R}^n)$ we let the function
$$F=F[\psi:\omega_1|V_1,\ldots,\omega_n|V_n]\hspace{10mm}F:{\s{T}_X}\to\bb{R}$$ be defined by
$$F(T):=\psi(\int_{T|V_1}\omega_1,\ldots,\int_{T|V_n}\omega_n).$$ Using the above properties of $T|V$ and similar with
(\ref{2110310701}) for any $v\in\r{Vec}_\r{c}(X)$ we have
\begin{equation*}\tilde{\r{d}}_v\big(F[\psi:\omega_1|V_1,\ldots,\omega_n|V_n]\big)=
F[\xi:\omega_1|V_1,\ldots,\omega_n|V_n,\r{d}_v\omega_1|V_1,\ldots,\r{d}_v\omega_n|V_n]\end{equation*} where $\xi:\bb{R}^{2n}\to\bb{R}$ is
defined by (\ref{2110300730}). Also similar with the results of $\S$\ref{2110300806} it is proved that the set $A$ of all functions
on ${\s{T}_X}$ of the forms $F[\psi:\omega_1|V_1,\ldots,\omega_n|V_n]$ is an algebra satisfying (C1) and (C2). Thus we have the lifted geometry
$(A,D)$ on ${\s{T}_X}$ where $D$ is given by (\ref{2110240850}).
\section{Action Functionals as Functions of Lifted Geometry}\label{2110300808}
In this section we consider a variant of the lifted geometry described in $\S$\ref{2110300805}. Let ${\s{C}_X}$ denote the set of
all smooth curves $C$ in $X$ defined on an arbitrary compact interval in $\bb{R}$. There is a canonical group-action given by
$$\r{Difeo}(X)\times{\s{C}_X}\to{\s{C}_X}\hspace{10mm}C\mapsto\theta\circ C.$$ For Lagrangian densities $L_1,\ldots,L_n$ on $X$ i.e. smooth
functions $L_i$ on the tangent bundle $\r{T}X$ of $X$, and any $\psi\in\r{C}^\infty(\bb{R}^n)$, let the \emph{generalized action functional}
$$F=F[\psi:L_1,\ldots,L_n]\hspace{10mm}F:{\s{C}_X}\to\bb{R}$$ be defined by
$$F(C):=\psi\big(\int_a^bL_1(C,\dot{C}),\ldots,\int_a^bL_n(C,\dot{C})\big)\hspace{10mm}(C:[a,b]\to X).$$
We show that for any $v\in\r{Vec}_\r{c}(X)$ and any Lagrangian density $L$ the directional derivative $\tilde{\r{d}}_vF_0$ of the action functional
$F_0:C\mapsto\int_a^bL(C,\dot{C})$ is equal to the action functional associated to a Lagrangian density which is the directional derivative of
$L$ along a vector field $v^\dag$ on the tangent bundle $\r{T}X$. We give the proof only in the simple case that $X=\bb{R}^k$. But using the concept
of prolongation of vector fields on jet bundles it can be stated in the general case. So suppose $L:\bb{R}^k\times\bb{R}^k\to\bb{R}$ and
$v:\bb{R}^k\to\bb{R}^k$ are smooth and $C:[a,b]\to\bb{R}^k$ is a curve. Using the linear approximation $\r{e}^{tv}(x)\sim x+tv(x)$ we have
\begin{equation*}\begin{split}
\big(\tilde{\r{d}}_vF_0\big)(C)&=\int_a^b\lim_{t\to0}\frac{L\big((\r{e}^{tv}\circ C)(s),(\r{e}^{tv}\circ C)'(s)\big)-L\big(C(s),C'(s)\big)}{t}\r{d}s\\
&=\int_a^b\lim_{t\to0}\frac{L\big(C(s)+tv(C(s)),C'(s)+tv'(C(s))(C'(s))\big)-L\big(C(s),C'(s)\big)}{t}\r{d}s\\
&=\int_a^b\big(\r{d}_{v^\dag}L\big)\big(C(s),C'(s)\big)\r{d}s\end{split}\end{equation*}
where $v^\dag:\bb{R}^k\times\bb{R}^k\to\bb{R}^k\times\bb{R}^k$ is defined by $$v^\dag\big(x,y\big):=\big(v(x),v'(x)(y)\big).$$
(In the above $v'(x):\bb{R}^k\to\bb{R}^k$ is the derivative of $v$ at $x$ and $C'(s)\in\bb{R}^k$ is the tangent vector to $C$ at $C(s)$.)
It follows that $$\tilde{\r{d}}_v\big(F[\psi:L_1,\ldots,L_n]\big)=F[\xi:L_1,\ldots,L_n,\r{d}_{v^\dag}L_1,\ldots,\r{d}_{v^\dag}L_n]$$ where
$\xi:\bb{R}^{2n}\to\bb{R}$ is defined by (\ref{2110300730}). Then it can be checked that
$$A:=\Big\{F[\psi:L_1,\ldots,L_n]:\psi\in\r{C}^\infty(\bb{R}^n),L_1,\ldots,L_n\in\r{C}^\infty(\r{T}\bb{R}^k),n\geq1\Big\}$$
is an algebra of functions on $\s{C}_{\bb{R}^k}$ satisfying (C1) and (C2). Thus $(A,D)$ may be regarded as a lifted
geometry for $\s{C}_{\bb{R}^k}$ where $D$ is given by (\ref{2110240850}).

\vspace{5mm}

\textbf{Conclusion:}
We defined the concept of Lifted Geometry and gave various examples and elementary applications of it.
It was clear that because of independence of any lifted geometry for an object from the existence of any topology or local
coordinate system on the object, Lifted Geometry becomes a tool to define differentiable structures on geometric objects with
infinite dimensional nature. There are many aspects of Lifted Geometry that needs to be explored and we have plan to do it in future works.
In our opinion the three concepts of flow, symmetry, and critical points of functions and vector fields in Lifted Geometry must have interesting applications in Mathematical Mechanics.


\end{document}